\documentclass[10pt]{amsart}

\usepackage{amssymb,amsmath,amscd,amsthm,amsfonts,wasysym,mathrsfs,amsxtra}
\input xy
\xyoption{all}

\newcommand{\CC}{\mathcal{C}}

\newcommand{\KK}{\mathcal{K}}

\newcommand{\EE}{\mathcal{E}}

\newcommand{\PP}{\mathbb{P}}

\newcommand{\image}{\textnormal{im}\,}
\newcommand{\kernel}{\textnormal{ker}\,}
\newcommand{\cokernel}{\textnormal{coker }}

\newcommand{\Hom}{\textnormal{Hom}}

\newcommand{\Ext}{\textnormal{Ext}}

\newcommand{\al}{\alpha}

\newcommand{\Coh}{\textnormal{Coh}}

\newcommand{\HH}{\mathcal{H}}

\newcommand{\arinj}{\ar@{^{(}->}}
\newcommand{\arsurj}{\ar@{->>}}

\newcommand{\module}{\textnormal{mod}\,}
\newcommand{\XX}{\mathcal{X}}
\newcommand{\YY}{\mathcal{Y}}
\newcommand{\BB}{\mathcal{B}}

\newtheorem{theorem}{Theorem}[section]

\newtheorem{lemma}[theorem]{Lemma}

\newtheorem{proposition}[theorem]{Proposition}
\newtheorem{coro}[theorem]{Corollary}

\theoremstyle{definition}

\theoremstyle{remark}
\newtheorem{remark}[theorem]{Remark}

\newtheorem{question}[theorem]{Question}

\numberwithin{equation}{section}

\begin{document}

\title[Torsion pairs and filtrations in abelian categories]{Torsion pairs and filtrations in abelian categories with tilting objects}

\author{Jason Lo}
\email{jccl@alumni.stanford.edu}


\begin{abstract}
Given a noetherian abelian category $\mathcal Z$ of homological dimension two with a tilting object $T$, the abelian category $\mathcal Z$ and the abelian category of modules over $\text{End} (T)^{\textit{op}}$ are related by a sequence of two tilts;  we give an explicit description of the torsion pairs involved.  We then use our techniques to obtain a simplified proof of a theorem of Jensen-Madsen-Su, that $\mathcal Z$ has a three-step filtration by extension-closed subcategories.  Finally, we generalise Jensen-Madsen-Su's filtration to a noetherian abelian category of any finite homological dimension.
\end{abstract}

\maketitle

\section{Introduction}

In representation theory and algebraic geometry, there are many important questions surrounding derived equivalences of the form
\begin{equation}\label{eq18}
  \Phi : D^b(\mathcal Z) \overset{\thicksim}{\to} D^b(\mathcal A)
\end{equation}
where $D^b(\mathcal Z)$ and $D^b(\mathcal A)$ are the bounded derived categories of two abelian categories $\mathcal Z$ and $\mathcal A$,  possibly with quite different origins.

For instance, in representation theory, we can take $\mathcal Z$ to be a module category, then let $T$ be a tilting object in $\mathcal Z$,  let $\mathcal A$ be the category of right modules over the endomorphism algebra of $T$, and let $\Phi$ be the derived functor $R\Hom (T,-)$.  In algebraic geometry, we can take $\mathcal Z$ to be the category of coherent sheaves on a variety $X$, let $Y$ be a moduli of stable sheaves on $X$, and let $\Phi$ be the Fourier-Mukai transform whose kernel is the universal family.  In both these two scenarios, we can try to deduce properties of moduli spaces of objects in $\mathcal Z$ from those of moduli spaces of objects in $\mathcal A$, or vice versa.  For moduli of modules, this strategy was  used in works such as Chindris' \cite{Chindris}.    For moduli of sheaves, the same strategy was used in papers by Bridgeland \cite{BriTh}, Bridgeland-Maciocia \cite{BMef}, Bruzzo-Maciocia \cite{BM}, and subsequent works by many others.

There are also occasions when $\mathcal Z$ is the category of coherent sheaves on a variety $X$, and $\mathcal A$ the module category over  the endomorphism algebra of a tilting sheaf $T$ on $X$, with $\Phi$ being the derived functor $R\Hom (T,-)$.  In this case, we obtain connections between moduli spaces of modules and moduli spaces of sheaves.  Results along this line of thought can be found in Craw \cite{Craw} and Ohkawa \cite{Ohk}, for example.  Even in the case when $T$ is not a tilting sheaf (in which case $\Phi$ is not necessarily an equivalence of derived categories), this approach still proves fruitful, as shown in \'{A}lvarez-C\'{o}nsul-King \cite{ACK}.

In many of the examples mentioned above, it helps to identify subcategories of $\mathcal Z$ that are `well-behaved', in the sense that  we may want them to be extension-closed, or for some of them to contain all the objects we  hope to parametrise in a moduli space.  In Fourier-Mukai transforms between varieties, these subcategories could be taken as the subcategories of `WIT$_i$-sheaves'; in exact equivalences between derived categories of modules, we can consider anologues of categories of WIT$_i$-sheaves, called the categories of `static' and `costatic' modules (e.g.\ see \cite[Section 4]{BB}  and \cite{Tonolo}).  Under our setting \eqref{eq18}, we denote these subcategories  by $\XX_i$ and $\YY_i$ in Section \ref{sec1} below.

In Chindris' work \cite{Chindris} (when the tilting object $T$ has homological dimension  one), a crucial property enjoyed by the subcategories $\XX_0, \XX_1$ of $\mathcal Z$ is that they `filter' the entire abelian category $\mathcal Z$.  That is, any object $E$ of $\mathcal Z$ has a filtration where the factors lie in the categories $\XX_i$.  Naturally, the question arises as to whether the categories $\XX_i$ filter $\mathcal Z$ when the homological dimension of $T$ is strictly larger than one.

As mentioned in Jensen-Madsen-Su \cite{JMS}, the categories $\XX_i$ are too small to filter $\mathcal Z$ even when $T$ has homological dimension two.  Jensen-Madsen-Su then constructs three subcategories $\EE_0, \EE_1, \EE_2$ of $\mathcal Z$,  containing the categories $\XX_0, \XX_1, \XX_2$, respectively, that are extension-closed and filter $\mathcal Z$ (see Theorem \ref{theorem1}).

In \cite{JMS}, it is remarked that their proofs of Theorems \ref{theorem1} and \ref{theorem2} can be simplified using spectral sequences.  We realise this spectral sequence approach in Section \ref{sec3},  recovering Theorems \ref{theorem1} and \ref{theorem2}.

On the other hand, motivated by mirror symmetry, it is important to understand the space of stability conditions (in the sense of Bridgeland) on the derived category $D^b(\mathcal Z)$, for $\mathcal Z$ of various origins, and the moduli spaces associated to these stability conditions.  One question within this framework is, given two t-structures on $D^b(\mathcal Z)$, are they related by a sequence of tilts (using torsion pairs)?  Answering this question can help us understand moduli spaces of stable objects in $D^b(\mathcal Z)$.  For instance, in Bayer-Macr\`{i}-Toda's work \cite{BMT}, they explicitly construct  a suitable t-structure for a conjectured stability condition on $D^b(\Coh(X))$ on a smooth projective threefold $X$; this t-structure was constructed by a sequence of two tilts from the standard t-structure.  Using the descriptions of these tilts from \cite{BMT}, we are able to describe some stable objects in $D^b(\Coh(X))$, as is done in \cite[Theorem 3.17]{LM}.

Under a derived  equivalence of the form \eqref{eq18}, we have two t-structures on $D^b(\mathcal Z)$, namely the standard t-structure, as well as the pullback of the standard t-structure on $D^b(\mathcal A)$ via $\Phi$.  It is not hard to show that these two t-structures are related by a sequence of two tilts when the homological dimension of $\mathcal Z$ is at most two - we do this in Proposition \ref{pro2}.  The point is, we explicitly describe the torsion pairs used in these tilts, and note the striking symmetry in diagram \eqref{eq15}.

In fact, in understanding the tilts between the two t-structures on $D^b(\mathcal Z)$ mentioned above, we are led to studying the subcategories $\BB_0, \BB_1, \BB_2$ of $\mathcal Z$ (defined in Section \ref{sec1}), which are larger than the categories $\XX_0, \XX_1, \XX_2$ when $T$ has homological dimension two.  The categories $\BB_i$ turn out to be a suitable tool for generalising Jensen-Madsen-Su's filtration to any noetherian abelian category with a tilting object $T$, for $T$ having any finite homological dimension.  Our generalisation is Theorem \ref{theorem3}; it reduces to Jensen-Madsen-Su's filtration  when $T$ has homological dimension two (see Corollary \ref{coro5}), and further reduces to the filtration used by Chindris \cite{Chindris} when $T$ has homological dimension one.


\bigskip
\noindent
\textbf{Acknowledgements.} The author is indebted to Calin Chindris for suggesting the problems that are addressed in this article, the many enlightening conversations without which the project could not have been completed, and various helpful comments on the exposition in the article.

\section{Notation}\label{sec1}


Throughout this article, $\mathcal Z$ will denote a noetherian abelian category of finite homological dimension $n \in \mathbb{Z}^+_0$ equipped with a tilting object $T$.  From now on, we will simply write $D(\mathcal Z)$ to denote the bounded derived category of $\mathcal Z$, and similarly for other abelian categories.

That the homological dimension  of $\mathcal Z$ is $n$ means, that  $\Ext^i_{\mathcal Z} (E,F)=0$ for all $E,F \in \mathcal Z$ and $i \notin [0,n]$.  That $T$ is a tilting object in $\mathcal Z$ implies that $\Ext^i_{\mathcal Z} (T,T)=0$ for all $i \neq 0$, and that the derived functor $\Phi := R\Hom_{D(\mathcal Z)} (T,-)$ induces a derived equivalence
\[
\xymatrix{
  D(\mathcal Z)  \ar[r]^\Phi_\thicksim & D (\mathcal A),
}
\]
where $\mathcal A := \module  A^{op}$, the category of finitely generated right $A$-modules, with $A$ being the endomorphism algebra $\Hom_{\mathcal Z}(T,T)$ of $T$.

Let $\Psi$ denote the quasi-inverse of $\Phi$, so that $\Psi \Phi \cong \text{id}$ and $\Phi \Psi \cong \text{id}$.  For any $E \in \mathcal Z$, we  will write $\Phi^i (E)$ to denote the degree-$i$ cohomology $H^i(\Phi (E))$ of $\Phi (E)$ with respect to the standard t-structure on $D(\mathcal A)$.  Similarly, for any $M \in D(\mathcal A)$, we will write $\Psi^j (M)$ to denote the degree-$j$ cohomology $H^j (\Psi M)$ of $\Psi (M)$ with respect to the standard t-structure on $D(\mathcal Z)$.

For any integer $i$, we define the full subcategory of $D(\mathcal Z)$
\[
  \XX_i^D := \{ E \in D(\mathcal Z) : \Phi^j (E)=0 \, \forall j \neq i\}.
\]
We also define the following subcategories of $\mathcal Z$
\begin{align*}
  \XX_i &:= \{ E \in \mathcal Z : \Phi^j(E)=0 \, \forall j \neq i\} =\XX_i^D \cap \mathcal Z, \\
  \BB_i &:= \{ E \in \mathcal Z : \Phi^i(E)=0\},
\end{align*}
as well as  the following subcategories of $\mathcal A$
\begin{align*}
  \YY_i &:= \{ M \in \mathcal A : \Psi^j (M)=0 \, \forall j \neq i\}=\Phi (\mathcal Z [-i]) \cap \mathcal A
  , \\
  \CC_i &= \{ M \in \mathcal A : \Psi^i (M)=0\}.
\end{align*}
And so, objects in the categories $\XX_i$ and $\YY_j$ are  analogues of the `WIT$_i$-sheaves' when we deal with Fourier-Mukai transforms between two varieites.

Note that $\Phi$ and $\Psi$ induce an equivalence of categories between $\XX_i$ and $\YY_{-i}[-i]$.

There are many properties of $\XX_i, \BB_i$ and $\YY_i$ that can be easily deduced from their definitions.  We will not list them all explicitly, except to note the following for now:
\begin{itemize}
\item $\XX_i, \XX_i^D,\BB_i, \YY_i, \CC_i$ are all closed under extensions.
\item $\BB_0$ is closed under subobjects in $\mathcal Z$.
\item $\BB_n$ is closed under quotients in $\mathcal Z$.
\end{itemize}
If $m$ is the largest integer for which $\mathcal C_{-m}$ is nonzero, then we also have
\begin{itemize}
\item $\mathcal C_0$ is closed under quotients in $\mathcal A$.
\item $\mathcal C_{-m}$ is closed under subobjects in $\mathcal A$.
\end{itemize}

While working in a fixed abelian category $\mathcal W$, for any subcategory $\mathcal V \subseteq \mathcal W$ we will write
\[
  \mathcal V^\circ := \{ E \in \mathcal W : \Hom_{\mathcal W} (\mathcal V, E)=0\}.
\]

\begin{remark}\label{remark6}
By \cite[Proposition 2.66]{FMTAG} and \cite[Section 3]{BB}, when $\Phi : D(\mathcal Z) \to D(\mathcal A)$ is a Fourier-Mukai transform and $\mathcal Z, \mathcal A$ are both categories of coherent sheaves on varieties, or when $\mathcal Z, \mathcal A$ are both module categories and $\Phi$ is the derived Hom functor comig from a tilting object,  we have the  spectral sequences ($\Psi$ being the quasi-inverse of $\Phi$)
\begin{equation}\label{eq1}
E_2^{p,q} = \Psi^p (\Phi^q E) \Rightarrow \begin{cases} E &\text{ if $p+q=0$} \\
0  &\text{ otherwise}\end{cases} \text{\quad for any $E \in \mathcal Z$},
\end{equation}
and
\begin{equation}\label{eq6}
  E^{p,q}_2 = \Phi^p(\Psi^q M) \Rightarrow \begin{cases} M &\text{ if $p+q=0$} \\ 0 &\text{ otherwise} \end{cases} \text{\quad for any $M \in \mathcal A$}.
\end{equation}
\end{remark}

The spectral sequences \eqref{eq1} and \eqref{eq6} will be used repeatedly when we recover Jensen-Madsen-Su's results below, so we will  say `the spectral condition holds' when these spectral sequences  exist.  When the spectral condition holds, given any object $E \in \mathcal Z$, we will simply write $E^{p,q}_r$ to denote the term $E^{p,q}_r$ in the spectral sequence \eqref{eq1} for $E$.  Similarly for objects in $\mathcal A$ and the spectral sequence \eqref{eq6}.

We will also write $\mathcal H := \Psi (\mathcal A)$.

\section{Tilts between the two hearts $\mathcal Z$ and $\mathcal A$}\label{sec2}

In this section, we consider the two hearts of t-structures $\mathcal Z$ and $\Psi (\mathcal A)$ in $D(\mathcal Z)$, and show that they are related by a sequence of two tilts in Proposition \ref{pro2}.  Moreover, we show that there is an apparent symmetry between these two tilts - see Lemma \ref{lemma10} and \eqref{eq15}.

The following lemma is instrumental in many constructions in this paper:

\begin{lemma}\cite[Lemma 1.1.3]{Pol}\label{lemma-Pol}
If $\mathcal A$ is a noetherian abelian category, then any full subcategory $\mathcal T$ of $\mathcal A$ closed under quotients and extensions is the torsion class of a torsion pair in $\mathcal A$.
\end{lemma}

Lemma \ref{lemma-Pol}, together with the observation that $\BB_n$ is closed under quotients and extensions in $\mathcal Z$, which we are assuming to be  noetherian, immediately gives:

\begin{coro}\label{coro1}
The pair $(\BB_n,\BB_n^\circ)$ is a torsion pair  in $\mathcal Z$.
\end{coro}

For basic properties of torsion pairs and tilting in an abelian category, the reader may refer to \cite[Chap. I, Sec. 2]{HRS}.

\subsection{When $\mathcal Z$ has homological dimension 1}

When  $\mathcal Z$ has homological dimension 1, we can regard the two hearts of t-structures  $\mathcal Z$ and  $\mathcal A$ as being related by a single tilt, as shown in the next lemma.

\begin{lemma}\label{lemma20}
Let $\mathcal Z$ be a noetherian abelian category of homological dimension 1.  Then
 \begin{itemize}
 \item[(a)] $\BB_1^\circ = \BB_0=\XX_1$, and $\BB_1=\XX_0$.
  \item[(b)] $\HH$ is the tilt of $\mathcal Z$ with respect to the torsion pair $(\BB_1, \BB_1^\circ)$.
  \end{itemize}
\end{lemma}

\begin{proof}
First, we show that $\mathcal B_1^\circ \subset \mathcal B_0$: take any $E \in \mathcal B_1^\circ$.  Since $T \in \mathcal B_1$, we have $\Hom (T,E)=0$, i.e.\ $E \in \mathcal B_0$.  Hence $\mathcal B_1^\circ \subset \mathcal B_0$ holds.

Next we show that $\mathcal B_0 \subset \mathcal B_1^\circ$.  Take any $F \in \mathcal B_0$, any $G \in \mathcal B_1$, and  any morphism $\al : G \to F$ in $\mathcal Z$.  Since $\mathcal B_0$ is closed under subobjects in $\mathcal Z$, we get $\image (\al) \in \mathcal B_0$.  On the other hand, after applying $\Phi = R\Hom (T,-)$ to the short exact sequence in $\mathcal Z$
\[
 0 \to \kernel (\al) \to G \to \image (\al) \to 0,
\]
part of the long exact sequence of cohomology is
\[
\Hom (T,G[1]) \to \Hom (T,\image (\al)[1]) \to \Hom (T,\kernel (\al)[2]).
\]
Since $\Hom (T,\kernel (\al)[2])=0$ (we are assuming $\mathcal Z$ to have homological dimension 1) and $G \in \mathcal B_1$, we get $\Hom (T,\image (\al)[1])=0$.  Hence $\Hom (T, \image (\al)[i])=0$ for all $i$, which forces $\image \al=0$, i.e.\ $\al$ is the zero map.  Hence $\mathcal B_0 \subset \mathcal B_1^\circ$, and we have $\mathcal B_1^\circ = \mathcal B_0$.  That $\mathcal B_0=\mathcal X_1$ and $\BB_1=\XX_0$ are clear.  This completes the proof of part (a).

To prove part (b),  let $\mathcal U_{1,1}$ denote the heart obtained by tilting $\mathcal Z$ with respect to $(\BB_1, \BB_1^\circ)$.  That is, $\mathcal U_{1,1} = \langle \BB_1^\circ[1], \BB_1\rangle$, i.e.\ $\mathcal U_{1,1} = \langle \XX_1 [1], \XX_0 \rangle$ by part (a).  Thus $\Phi (\mathcal U_{1,1}) \subset \mathcal A$. Since both $\Phi (\mathcal U_{1,1})$ and $\mathcal A$ are hearts of bounded t-structures, they must be equal, i.e.\ $\mathcal U_{1,1} = \mathcal H$, proving part (b).
\end{proof}

\begin{remark}\label{remark3}
When $X$ is a smooth projective curve, $\mathcal Z = \Coh (X)$ and $T$ is a tilting sheaf on $X$, Lemma \ref{lemma20} says that the two hearts $\Coh (X)$ and $\module A^{op}$ (where $A$ is the endomorphism algebra of $T$) differ by a single tilt.
\end{remark}

\subsection{When $\mathcal Z$ has homological dimension 2}

When $\mathcal Z$ has homological dimension 2 and  the spectral condition holds, we will again show that $\mathcal Z$ and $\mathcal A$ are related by tilting; this time, they are related by a sequence of two tilts.

\begin{lemma}\label{lemma6}
Suppose $\mathcal Z$ is a noetherian abelian category of homological dimension 2. Then $\BB_2^\circ = \BB_0 \cap \XX_1^\circ$, and so $(\BB_2,\BB_0 \cap \XX_1^\circ)$ is a torsion pair in $\mathcal Z$.
\end{lemma}

\begin{proof}
Take any $E \in \BB_2^\circ$.  Since $T \in \BB_2$, we have $\Hom (T,E)=0$, i.e.\ $E \in \BB_0$.  Hence $\BB_2^\circ \subseteq \BB_0$.  On the other hand, that $\XX_1 \subseteq \BB_2$ implies $\BB_2^\circ \subseteq \XX_1^\circ$.  Hence $\BB_2^\circ \subseteq \BB_0 \cap \XX_1^\circ$.

To show the other inclusion, take any $F \in \BB_0 \cap \XX_1^\circ$, any $G \in \BB_2$ and any morphism $\al : G \to F$ in $\mathcal Z$.  We have the exact sequence in $\mathcal Z$
\begin{equation*}
0 \to \kernel (\al) \to G \overset{\al}{\to} F \to \cokernel (\al) \to 0.
\end{equation*}
Since $\BB_2$ is closed under quotients in $\mathcal Z$ while $\BB_0$ is closed under subobjects in $\mathcal Z$, we have  $\image (\al) \in \BB_0 \cap \BB_2$, i.e.\ $\image (\al) \in \XX_1$.  However, $F \in \XX_1^\circ$, so $\al$ must be the zero morphism.  This completes the proof that $\BB_2^\circ = \BB_0 \cap \XX_1^\circ$.  By Corollary \ref{coro1}, we see that $(\BB_2, \BB_0 \cap \XX_1^\circ)$ is a torsion pair in $\mathcal Z$.
\end{proof}

\begin{proposition}\label{pro2}
Let $\mathcal U_{2,1}$ denote the heart obtained by tilting $\mathcal Z$ with respect to the torsion pair $(\BB_2, \BB_0 \cap \XX_1^\circ)$.  Then $\HH$ can be obtained from $\mathcal U_{2,1}$ by a tilt with respect to a torsion pair in $\mathcal U_{2,1}$.
\end{proposition}

\begin{proof}
Since $\mathcal U_{2,1} = \langle (\BB_0 \cap \XX_1^\circ) [1], \BB_2 \rangle$, we see that $\Phi (\mathcal U_{2,1}) \subseteq \langle \mathcal A, \mathcal A [-1]\rangle$.  Thus, by Proposition \ref{pro1} below, we know $\Phi (\mathcal U_{2,1})$ is a tilt of $\mathcal A [-1]$, i.e.\ $\mathcal U_{2,1}$ is a tilt of $\HH [-1]$ with respect to the torsion pair $(\mathcal T, \mathcal F)$ in $\HH [-1]$, where
\begin{align}
  \Phi \mathcal T &= \mathcal A[-1] \cap \Phi (\mathcal U_{2,1}), \notag\\
  \Phi \mathcal F &= \mathcal A[-1] \cap \Phi (\mathcal U_{2,1})[-1], \label{eq2}
\end{align}
i.e.\
\begin{align*}
 \mathcal T &= \mathcal U_{2,1} \cap (\HH [-1]), \\
 \mathcal F &= (\mathcal U_{2,1}\cap \HH)[-1].
\end{align*}
\end{proof}

\begin{proposition}\cite[Proposition 2.3.2(b)]{BMT}\label{pro1}
If $\mathcal A, \mathcal B$ are the hearts of two bounded t-structures on a triangulated category $\mathcal D$, and $\mathcal B \subset \langle \mathcal A, \mathcal A [1]\rangle$, then
\begin{equation*}
  \mathcal T := \mathcal A \cap \mathcal B \text{\quad and \quad}  \mathcal F := \mathcal A \cap \mathcal B [-1]
   \end{equation*}
    form a torsion pair in $\mathcal A$, and $\mathcal B$ is the tilt of $\mathcal A$ with respect to the torsion pair $(\mathcal T, \mathcal F)$.
\end{proposition}

When the spectral condition holds, we can say a little more about the cohomology objects $\Phi^0 E$ and $\Phi^2E$ for any $E \in \mathcal Z$:

 \begin{lemma}\label{lemma7}
 Suppose $\mathcal Z$ is a noetherian abelian category of homological dimension 2 and  the spectral condition holds.  Then for any $E \in \mathcal Z$, we have $\Phi^0(E) \in \YY_0$ and $\Phi^2E \in \YY_{-2}$.  If $E \in \BB_0$, then $\Psi^{-2}(\Phi^1 E)=0$.
 \end{lemma}


\begin{proof}
Consider the spectral sequence \eqref{eq1} for $E \in \mathcal Z$.  Since $\mathcal Z$ has homological dimension 2, along with \cite[Lemma 5]{JMS}, we have $E_2^{p,q}=0$ unless $-2 \leq p \leq 0$ and   $0\leq q \leq 2$.  That $E_2^{p,q} \Rightarrow 0$ for $p+q \neq 0$ implies that $E_2^{-2,0}, E_2^{-1,0}, E_2^{-1,2}$ and $E_2^{0,2}$ all vanish, giving us the first claim.  If $E \in \mathcal B_0$, then $E_2^{-2,1} = E_\infty^{-2,1}$ also vanishes.
\end{proof}

Similarly, we have:

\begin{lemma}\label{lemma9}
Under the hypotheses of Lemma \ref{lemma7},  for any $M \in \mathcal A$, $\Psi^0(M) \in \XX_0$ and $\Psi^{-2}(M) \in \XX_2$.
\end{lemma}


As a result, we can write the torsion pair in Lemma \ref{lemma6} in a different way:

\begin{lemma}\label{lemma8}
Suppose $\mathcal Z$ is a noetherian abelian category of homological dimension 2 and the spectral condition holds.  Then  $\BB_0 = \XX_0^\circ$, and so $(\BB_2, \XX_0^\circ \cap \XX_1^\circ)$ is a torsion pair in $\mathcal Z$.
\end{lemma}

\begin{proof}
That $\BB_0 \subseteq \XX_0^\circ$ is clear.  Given any $E \in \mathcal Z$, we have   $\Phi^0 E \in \YY_0$ by Lemma \ref{lemma7}.  If we apply $\Psi$ to the canonical exact triangle in $D(\mathcal A)$
\[
\Phi^0 E \to \Phi E \to \tau^{\geq 1}(\Phi E),
\]
we get the exact triangle in $D(\mathcal Z)$
\[
 \Psi (\Phi^0E) \to E \to \Psi (\tau^{\geq 1}(\Phi E))
\]
where $\Psi (\Phi^0 E) \in \XX_0$ by Lemma \ref{lemma7}.  If $E \in \XX_0^\circ$, then $\Phi^0 E$ must be zero, i.e.\ $E \in \BB_0$.  Hence $\XX_0^\circ \subseteq \BB_0$.  By Lemma \ref{lemma6}, we are done.
\end{proof}

The following lemma gives another description of the tilt from $\mathcal U_{2,1}$ to $\mathcal H$ or, equivalently, from $\mathcal A$ to $\Phi (\mathcal U_{2,1})[1]$, which is perhaps more illuminating than the description in Proposition \ref{pro2}.

%

\begin{lemma}\label{lemma10}
The heart $\Phi (\mathcal U_{2,1})[1]$ can be obtained from $\mathcal A$ by tilting with respect to the torsion pair
\[
(\Phi\mathcal T[1], \Phi \mathcal F[1]) = (\CC_0, \YY_{-1}^\circ \cap \YY_{-2}^\circ).
\]
\end{lemma}

Putting Proposition \ref{pro2} and Lemmas \ref{lemma10} together, we can  summarise the tilts we have constructed so far in the following diagram, where the left column represents the tilt on the $D(\mathcal Z)$ side, and the right column represents the tilt on the $D(\mathcal A)$ side:

\begin{equation}\label{eq15}
\xymatrix{
\mathcal Z \ar[d]^{(\BB_2, \XX_0^\circ \cap \XX_1^\circ)} & \\
\mathcal U_{2,1} & \Phi (\mathcal U_{2,1})[1] \\
& \mathcal A \ar[u]_{(\CC_0,\YY_{-1}^\circ \cap \YY_{-2}^\circ)}
}
\end{equation}

\begin{remark}\label{remark4}
From diagram \eqref{eq15}, it is as if the tilt on the left is a `mirror image' of the tilt on the right. However, it is not clear whether this phenomenon holds in general.
\end{remark}


\begin{proof}[Proof of Lemma \ref{lemma10}]
From the proof of Proposition \ref{pro2}, we already know that $\Phi (\mathcal U_{2,1})[1]$ can be obtained by tilting $\mathcal A$ at the torsion pair $(\Phi \mathcal T [1], \Phi \mathcal F [1])$.  So what we want to show here are $\Phi \mathcal T [1] = \CC_0$ and $\Phi \mathcal F [1] = \YY_{-1}^\circ \cap \YY_{-2}^\circ$.

To start with, let us show $\Phi \mathcal T [1] = \CC_0$.  Take any $M \in \Phi \mathcal T [1]$.  Then $\Psi M \in \mathcal U_{2,1}[1] \subseteq D^{[-2,-1]}_{\mathcal Z}$, and so $\Psi M \in \CC_0$.  For the other inclusion, take any $M \in \CC_0$.  Then $\Psi^0M =0$, and from the spectral sequence \eqref{eq6}, we see that $\Psi^{-1}M \in \BB_2$.  Also, $\Psi^{-2}M$ lies in $\XX_2$, and so lies in $\XX_0^\circ \cap \XX_1^\circ$.  Overall, $\Psi M \in \mathcal U_{2,1}[1]$ by Lemma \ref{lemma8}, giving us $M \in \Phi (\mathcal U_{2,1})[1] \cap \mathcal A = \Phi \mathcal T [1]$.  Hence $\Phi \mathcal T [1] = \CC_0$.

Next, let us show $\Phi \mathcal F [1] = \YY_{-1}^\circ \cap \YY_{-2}^\circ$.  Take any $M \in \Phi \mathcal F [1]$, any $G \in \YY_{-1}$ and any $A$-linear map $\al : G \to M$.  Let $I := \image (\al)$.  Since $\Phi \mathcal F [1]$ is the torsion-free class in a torsion pair in $\mathcal A$, it is closed under taking subobjects in $\mathcal A$.  Hence $I \in \Phi \mathcal F [1]$, and so $\Psi (I) \in \mathcal U_{2,1} \cap \XX_0^D$.  On the other hand, $\Psi G \in \XX_1[1]$.  Writing $\bar{\al}$ for the surjection $G \to I$ in $\mathcal A$ induced by $\al$, we see that $\Psi \bar{\al} \in \Hom_{\mathcal Z}(\Psi G, \Psi I)$ is  induced by some $\al' \in \Hom_{\mathcal A} (\Psi^{-1}G, \Psi^{-1}I)$ (note: here, we are using $\Psi^i$ to denote the functor $H^i \circ \Psi$).  However, $\Psi^{-1}G \in \XX_1$ while $\Psi^{-1}I \in \XX_1^\circ$, so $\al'$ must be zero, i.e.\ $\al$ must be zero.  This shows that $M \in \YY_{-1}^\circ$.  That $M \in \YY_{-2}^\circ$ is clear ($\Psi M$ has cohomology at degrees $-1$ and 0 only).  Hence $\Phi \mathcal F [1] \subseteq \YY_{-1}^\circ \cap \YY_{-2}^\circ$.

To prove the other inclusion, take any $N \in \YY_{-1}^\circ \cap \YY_{-2}^\circ$.  We want to show that $N \in \Phi \mathcal F [1]$, which is equivalent to  $N \in \mathcal A$ (which is clear) and $\Psi N \in \mathcal U_{2,1}$ by \eqref{eq2}.  If $\Psi^{-2}N \neq 0$, then the canonical morphism $\Psi^{-2}N [2] \to \Psi N$ gives a nonzero morphism $\Phi^2 (\Psi^{-2}N) \to N$ in $\mathcal A$, which is impossible since $\Phi^2 (\Psi^{-2} N) \in \YY_{-2}$.  Hence $\Psi^{-2}N =0$, and we have an exact triangle in $\mathcal Z$
\begin{equation}\label{eq16}
  \Psi^{-1}N [1] \to \Psi N \to \Psi^0 N.
\end{equation}
Now, $\Psi^0 N$ lies in $\XX_0$ by Lemma \ref{lemma9}, and in particular $\Psi^0 N \in \BB_2$.  Therefore, it remains to show that $\Psi^{-1}N \in \BB_0 \cap \XX_1^\circ$ (by Lemma \ref{lemma8}).

Applying $\Phi$ to the exact triangle \eqref{eq16}, we see that $\Phi^0 (\Psi^{-1} N)=0$, and so $\Psi^{-1}N \in \BB_0$.  Finally, to show $\Psi^{-1}N \in \XX_1^\circ$, let us take any $G \in \XX_1$.  Since we have the isomorphism
\[
  \Hom_{D(\mathcal Z)} (G[1],\Psi N) \cong \Hom_{\mathcal Z} (G,\Psi^{-1}N),
\]
any nonzero morphism $\theta : G \to \Psi^{-1}N$ induces a nonzero morphism $\bar{\theta} : G [1] \to \Psi N$, and hence a nonzero morphism $\Phi (\bar{\theta}) : \Phi^1 G \to N$.  However, $\Phi^1 G \in \YY_{-1}$ while $N \in \YY_{-1}^\circ$, so we have a contradiction.  Hence $\Psi^{-1}N$ must lie in $\XX_1^\circ$, and this completes the proof of this lemma.
\end{proof}


\begin{remark}
As in Remark \ref{remark3}, if $X$ is a smooth projective surface, $\mathcal Z = \Coh (X)$ and $T$ a tilting sheaf on $X$, then Proposition \ref{pro2} says that the two hearts $\Coh (X)$ and $\module A^{op}$ are related by a sequence of two tilts.
\end{remark}

\begin{remark}
    When $\mathcal Z$ is the category $\Coh (X)$ of coherent sheaves on a smooth projective variety $X$ over $\mathbb{C}$, there is an infinite number of simple objects in $\Coh (X)$ (e.g.\ the skyscraper sheaves), and so Assumption 1 in Woolf \cite{Woolf} is not satisfied.
\end{remark}

Since $\CC_0$ is closed under extensions and quotients in $\mathcal A$, Lemma \ref{lemma-Pol} tells us that we have a torsion pair $(\CC_0,\CC_0^\circ)$ in $\mathcal A$.  Combining this observation with Lemma \ref{lemma10}, we obtain:

\begin{coro}\label{coro4}
Suppose $\mathcal Z$ is a noetherian abelian category of homological dimension 2 and the spectral condition holds.  Then  $\CC_0^\circ = \YY_{-1}^\circ \cap \YY_{-2}^\circ$.
\end{coro}

Given the results in this section, it is natural to ask:

\begin{question}\label{question1}
Suppose $\mathcal Z$ is  the category of coherent sheaves on a smooth projective variety of dimension $n$ (resp.\  the category of finitely generated modules over a finite-dimensional algebra of homological dimension $n$),  while $T$ is a tilting sheaf (resp.\ a tilting object), and $A$ is the endomorphism algebra of $T$.  Are $\mathcal Z$ and $\module A^{op}$ related by a sequence of $n$ tilts?
\end{question}


\section{Results of Jensen-Madsen-Su}\label{sec3}

In this section, we give simplified  proofs of the two main theorems in \cite{JMS} by using the spectral sequences \eqref{eq1} and \eqref{eq6}.  For convenience, throughout this section, $\mathcal Z$ will be a noetherian abelian category of homological dimension two that is $k$-additive for a field $k$ and Hom-finite, and we will assume that the  spectral condition holds.

As in \cite{JMS}, we define the following full subcategories of $\mathcal Z$:
\begin{align*}
    \KK_0 &:= \{\text{cokernels of injections from objects in $\XX_2$ to objects in $\XX_0$}\}\\
    \KK_2 &:= \{ \text{kernels of surjections from objects in $\XX_2$ to objects in $\XX_0$}\} \\
    \KK_1 &:= \XX_1 \\
    \EE_i &:= \text{extension-closure of $\KK_i$ for $i=0,1,2$}.
\end{align*}
Note the following relations:
\begin{equation*}
  \XX_0 \subseteq \KK_0 \subseteq  \BB_2, \, \XX_2 \subseteq \KK_2 \subseteq \BB_0, \, \XX_1= \KK_1 =\EE_1 .
  \end{equation*}

The two main theorems of Jensen-Madsen-Su in \cite{JMS} are as follows:

\begin{theorem}\cite[Theorem 2]{JMS}\label{theorem1}
Suppose $T$ is a tilting object in $\mathcal Z$, which has homological dimension at most two.  Then for any object $E \in \mathcal Z$, there is a unique and functorial filtration
\begin{equation}\label{eq3}
0 = E_0 \subseteq E_1 \subseteq E_2 \subseteq E_3 = E
\end{equation}
with $E_{i+1}/E_i \in \mathcal E_i$, where the $\mathcal E_i$  are pairwise disjoint and extension-closed subcategories of $\mathcal Z$.
\end{theorem}

The filtration \eqref{eq3} can be refined as shown below:

\begin{theorem}\cite[Theorem 4]{JMS} \label{theorem2}
Suppose $T$ is a tilting object in $\mathcal Z$, which has homological dimension at most two.  Then for any object $E \in \mathcal Z$, there is a filtration
\begin{equation}\label{eq8}
0 = Z_0 \subseteq \cdots \subseteq Z_n \subseteq Y_n \subseteq \cdots \subseteq Y_0 = E
\end{equation}
with all $Z_{i+1}/Z_i \in \KK_0$, $Y_n/Z_n \in \KK_1$ and all $Y_i/Y_{i+1} \in\KK_2$, for some $n$.
\end{theorem}

Note that, Theorem \ref{theorem2} implies the existence of the filtration in Theorem \ref{theorem1} simply from the way we define the categories $\EE_i$.


\begin{lemma}\cite[Lemma 12]{JMS}\label{lemma11}
For any object $E \in \mathcal Z$, we have a filtration in $\mathcal Z$
\begin{equation}\label{eq10}
  F^0 E \subseteq F^{-1}E \subseteq F^{-2}E = E
\end{equation}
where $F^0E \in \KK_0, F^{-1}E/F^0E \cong E^{-1,1}_2$ and $F^{-2}E/F^{-1}E \in \KK_2$.
\end{lemma}

\begin{proof}
From the spectral sequence  \eqref{eq1}, we have a filtration
\begin{gather*}
  F^0E \subseteq F^{-1}E \subseteq F^{-2}E = E \text{ where } \\
   F^0E \cong E^{0,0}_\infty, \, F^{-1}E/F^0E \cong E^{-1,1}_2 \cong E^{-1,1}_\infty, \, F^{-2}E/F^{-1}E \cong E^{-2,2}_\infty.
\end{gather*}
From the $E_2$ page of the spectral sequence, we obtain  short exact sequences in the abelian category $\mathcal Z$
\begin{align}
  0 \to E^{-2,1}_2 \to E^{0,0}_2 \to E^{0,0}_\infty \to 0, \label{eq4} \\
    0 \to E^{-2,2}_\infty \to E^{-2,2}_2 \to E^{0,1}_2 \to 0 \label{eq5}.
\end{align}
By Lemma \ref{lemma9}, the short exact sequence \eqref{eq4} immediately gives us $E^{0,0}_\infty \in \KK_0$, while the short exact sequence \eqref{eq5} immediately gives   $E^{-2,2}_\infty \in \KK_2$.
\end{proof}


For any object $E \in \mathcal Z$, let $d(E)$ denote the dimension of $\Phi^1 E$ as a vector space over $k$.  Lemma \ref{lemma11}, together with the following lemma, will give us Theorem \ref{theorem2}:

\begin{lemma}\cite[Lemma 13]{JMS}\label{lemma12}
For any  $E \in \mathcal Z$, we have $d(E^{-1,1}_2) \leq d(E)$, and equality holds only if $E^{-1,1}_2 \in \XX_1$.
\end{lemma}

%
%
%


\begin{proof}[Proof of Theorem \ref{theorem2}]
Take any $E \in \mathcal Z$.  Let $Z_0=0$ and $Y_0 = E$.  By Lemma \ref{lemma11}, we have a filtration of $E$
\begin{equation}\label{eq9}
0 = Z_0 \subseteq Z_1 \subseteq Y_1 \subseteq Y_0 = E
\end{equation}
where $Z_1 := F^0E$ and $Y_1 := F^{-1}E$ as in \eqref{eq10}; Lemma \ref{lemma11} also tells us that $Z_1/Z_0 \in \KK_0, Y_1/Z_1 \cong E^{-1,1}_2$ and $Y_0/Y_1 \in \KK_2$.

We can now repeatedly apply Lemma \ref{lemma11} to $Y_i/Z_i$ for $i \geq 0$ to refine the filtration \eqref{eq9}.  By Lemma \ref{lemma12}, there exists an $n$ such that $Y_n/Z_n \in \XX_1$, at which point we have constructed the desired filtration \eqref{eq8}.
\end{proof}

\begin{remark}\label{remark1}
As noted earlier, Theorem \ref{theorem2} implies the existence of the filtration in Theorem \ref{theorem1}.  In fact, we also have functoriality for the filtration \eqref{eq8} in Theorem \ref{theorem2}: since the filtration \eqref{eq8} comes from the spectral sequences \eqref{eq1} and \eqref{eq6}, the functoriality of \eqref{eq8} follows from that of the Cartan-Eilenberg resolution (see the proof of \cite[Proposition 2.66]{FMTAG}).
\end{remark}


Recall from Corollary \ref{coro1} that we have a torsion pair $(\BB_2, \BB_2^\circ)$ in $\mathcal Z$, and so for  any $E \in \mathcal Z$, we have a filtration of the form
\begin{equation}\label{eq11}
 0 \to E_T \to E \to E_F \to 0
\end{equation}
where $E_T \in \BB_2$ and $E_F \in \BB_2^\circ$.  From torsion theory, we know that such a filtration is unique.  The following lemma says that the filtration constructed by Jensen-Madsen-Su can be obtained as a refinement of the filtration \eqref{eq11}:

\begin{lemma}\label{lemma14}
For any $E \in \mathcal Z$, we can refine the filtration \eqref{eq11} of $E$ to obtain a filtration of the form \eqref{eq3}, in the following sense:
 \begin{itemize}
 \item the term $E_F$ in \eqref{eq11} lies in $\EE_2$;
 \item the term $E_T$ in \eqref{eq11} is an extension
 \[
  0 \to E_{T,1} \to E_T \to E_{T,2} \to 0 \text{ in $\mathcal Z$}
 \]
 where $E_{T,1} \in \EE_0, E_{T,2} \in \EE_1$.
 \end{itemize}
\end{lemma}

We single out a step in the proof of Lemma \ref{lemma14}:
\begin{lemma}\label{lemma15}
We have $\EE_2 = \BB_2^\circ$.
\end{lemma}


\begin{proof}
Take any nonzero $E \in \KK_2$.  By definition, $E$ is the kernel of some surjection $\al : G \to B$ in $\mathcal Z$ where $G \in \XX_2, B \in \XX_0$.  If there is a nonzero morphism $\beta : C \to E$ where $C \in \BB_2$, then the induced map $C \to G$ would be a nonzero morphism.  However, $\Hom_{D(\mathcal A)}(\Phi C, \Phi G)=0$, so we have a contradiction. This shows that $\KK_2 \subseteq \BB_2^\circ$, and so $\EE_2 \subseteq \BB_2^\circ$.

To show the other inclusion, i.e.\ $\BB_2^\circ \subseteq \EE_2$, let us take any $E \in \BB_2^\circ$.  Clearly, $\Hom (\XX_1, E)=0$.  We claim that we also have $\Hom (\KK_0, E)=0$: for any nonzero $G \in \KK_0$, we have a surjection $G' \twoheadrightarrow G$ in $\mathcal Z$ where $G' \in \XX_0 \subseteq \BB_2$.  Since any composite morphism $G' \twoheadrightarrow G \to E$ must be zero, we see that $\Hom (G,E)=0$.  Hence $\Hom (\KK_0, E)=0$, and by Theorem \ref{theorem2}, $E$ itself must lie in $\EE_2$.
\end{proof}

\begin{proof}[Proof of Lemma \ref{lemma14}]
Consider the filtration \eqref{eq11} of $E$.  By Lemma \ref{lemma15}, we have $E_F \in \EE_2$.  Since $E_T \in \BB_2$ and $\Hom (\BB_2, \EE_2)=0$ (also by Lemma \ref{lemma15}), Theorem \ref{theorem2} tells us that $E_T$ has a filtration in $\mathcal Z$ of the form
\[
  0 \subseteq E_{T,1} \subseteq E_T \text{ in $\mathcal Z$}
\]
where $E_{T,1} \in \EE_0$ and $E_T/E_{T,1} \in \EE_1$.
\end{proof}

In Section \ref{sec4}, we give a generalisation of the filtration \eqref{eq3} in Theorem \ref{theorem1}; the following lemma will follow immediately from Theorem \ref{theorem3} and Corollary \ref{coro5} (since every torsion pair gives a unique two-step filtration):

\begin{lemma}\label{lemma16}
The filtration \eqref{eq3} is unique.
\end{lemma}

%

Theorem \ref{theorem3} and Corollary \ref{coro5} also tell us:

\begin{coro}\label{coro2}
The categories $\EE_0, \EE_1, \EE_2$ have trivial pairwise intersections.
\end{coro}

We have now recovered Theorem \ref{theorem1} of Jensen-Madsen-Su in its entirety:

\begin{proof}[Proof of Theorem \ref{theorem1}]
  Remark \ref{remark1} already explained why we have the existence and functoriality parts of the theorem.  The uniqueness part was Lemma \ref{lemma16}.
\end{proof}



As remarked at the end of \cite[Section 1]{JMS}, using spectral sequences gives a relatively efficient proof of Theorems \ref{theorem1} and \ref{theorem2}.  As pointed out in \cite{JMS} as well, however, it seems much more difficult to generalise the spectral sequence argument in this section to the case where  $\mathcal Z$ has homological dimension higher than two.  In other words, for higher homological dimensions, it is not clear how to define analogues of the categories $\EE_0, \EE_1, \EE_2$ using the spectral sequences \eqref{eq1} and \eqref{eq6}.

We end this section with some easy observations and  speculations on generalising Theorems \ref{theorem1} and \ref{theorem2} using spectral sequences:


\begin{lemma}\label{lemma21}
Suppose $\mathcal Z$ is a noetherian abelian category of homological dimension $n \geq 1$, and that $\Psi^j G=0$ for all $G \in \mathcal A$ and $j \notin [-n,0]$.  Let $E \in \mathcal Z$, and suppose the filtration of $E$ given by the spectral sequence \eqref{eq1} is
\begin{equation}\label{eq17}
  0=F^1E \subseteq F^0E \subseteq F^{-1}E \subseteq \cdots \subseteq F^{-n}E = E
\end{equation}
where $F^iE/F^{i+1}E \cong E^{i,-i}_\infty$ for $-n \leq i \leq 0$.  Then $E^{0,0}_\infty \in \BB_n$, and $E^{-n,n}_\infty \in \BB_0$.
\end{lemma}

\begin{proof}
Observe that we have a series of surjections in $\mathcal Z$
\[
  E^{0,0}_2 \twoheadrightarrow E^{0,0}_3 \twoheadrightarrow \cdots \twoheadrightarrow E^{0,0}_n \twoheadrightarrow E^{0,0}_\infty.
\]
Since $E^{0,0}_2 \in \BB_n$ (this follows from the spectral sequence \eqref{eq6}) and $\BB_n$ is closed under quotients in $\mathcal Z$, we have $E^{0,0}_\infty \in \BB_n$.  The proof for the second part is similar: we use the series of injections in $\mathcal Z$
\[
   E^{-n,n}_\infty \hookrightarrow E^{-n,n}_n \hookrightarrow \cdots \hookrightarrow E^{-n,n}_2,
\]
and  note that $E^{-n,n}_2 \in \BB_0$ and that $\BB_0$ is closed under taking subobjects in  $\mathcal Z$.
\end{proof}

\begin{remark}\label{remark5}
In fact, in the proof of Lemma \ref{lemma21} above, for any object $E \in \mathcal Z$, the spectral sequence \eqref{eq6} gives us $E^{0,0}_2 \in \BB_{n-1} \cap \BB_n$ and $E^{-n,n}_2 \in \BB_0 \cap \BB_1$.  Using the notation in \eqref{eq19}, we have $E^{0,0}_\infty \in [\BB_{n-1} \cap \BB_n]$.
\end{remark}

Judging from the way we used Lemmas \ref{lemma11} and \ref{lemma12} to obtain the filtration \eqref{eq8}, one might guess that, when the homological dimension of $\mathcal Z$ is higher than two, we could repeatedly filter the intermediate terms $\{F^iE/F^{i+1}E\}_{-n \leq i \leq 0}$ in \eqref{eq17} until they stabilise as in Lemma \ref{lemma12}.

\section{A generalisation of Jensen-Madsen-Su's filtration}\label{sec4}



In this section, we give a generalisation of Jensen-Madsen-Su's filtration of the category $\mathcal Z$, for any noetherian abelian category $\mathcal Z$ having any finite homological dimension.  For this generalisation, we do not assume that the spectral condition holds.  The idea is to produce $n$ torsion pairs when $\mathcal Z$ has homological dimension $n$, and take intersections of these torsion and torsion-free classes.  These intersections will be extension-closed subcategories of $\mathcal Z$ that contain the factors of the filtration for any $E \in \mathcal Z$.  To produce these torsion pairs in $\mathcal Z$, we need:

\begin{lemma}\label{lemma17}
Let $\mathcal Z$ be an abelian category, and $\mathcal S$ any full subcategory of $\mathcal Z$. Let $[S]$ denote the extension-closure generated by $\mathcal Z$-quotients of objects in $S$, i.e.\
\begin{equation}\label{eq19}
  [S] := \langle \{ E \in \mathcal Z : \exists \, E' \twoheadrightarrow E \text{ in $\mathcal Z$}, E' \in S\}\rangle.
\end{equation}
Then  $[S]$ is closed under quotients in $\mathcal Z$.
\end{lemma}

\begin{proof}
Take any object $E \in [S]$.  Then we have a filtration in $\mathcal Z$
\[
 0=E_0 \subseteq E_1 \subseteq E_2 \subseteq \cdots \subseteq E_m = E
\]
along with some surjections $G_i \twoheadrightarrow E_i/E_{i-1}$ in $\mathcal Z$, for  $1 \leq i \leq m$.  For any surjection $\phi : E \to E'$ in $\mathcal Z$, we have the following filtration for $E'$:
\[
  0 \subseteq \phi (E_1) \subseteq \phi (E_2) \subseteq \cdots \subseteq \phi (E_m)= E',
\]
where each $\phi (E_i)$ denotes the image of $E_i$ under $\phi$; we also have induced surjections $E_i/E_{i-1} \twoheadrightarrow \phi (E_i)/\phi (E_{i-1})$.  This completes the proof of the lemma.
\end{proof}

\begin{coro}\label{coro3}
Let $\mathcal Z$ be a noetherian abelian category, and $S$ any full subcategory of $\mathcal Z$. Then we have a torsion pair $([S], [S]^\circ)$ in $\mathcal Z$.
\end{coro}
\begin{proof}
By Lemma \ref{lemma17},  $[S]$ is closed under taking quotients and extensions in $\mathcal Z$.  The corollary then follows from Lemma \ref{lemma-Pol}.
\end{proof}

%


For the rest of this section, assume that $\mathcal Z$ is a noetherian abelian category of homological dimension $n$, where $n$ is finite.  For any $1 \leq i \leq n$, let us define
\begin{align*}
  \mathcal T_i &:= [\BB_i \cap \cdots \cap \BB_{n-1} \cap \BB_n], \\
  \mathcal F_i &:= \mathcal T_i^\circ = \{ E \in \mathcal Z : \Hom_{\mathcal Z}(\mathcal T_i,E)=0\}.
\end{align*}
We also define $\mathcal T_{n+1} = \mathcal Z$ and $\mathcal F_{n+1}=\{0\}$.  By Corollary \ref{coro3}, we have a torsion pair $(\mathcal T_i, \mathcal F_i)$ in $\mathcal Z$ for each $1\leq i \leq n+1$.

We  can now prove a generalisation of Jensen-Madsen-Su's Theorem \ref{theorem1}:

\begin{theorem}\label{theorem3}
Let $\mathcal Z$ be a noetherian abelian category of finite homological dimension $n$. Given any $E \in \mathcal Z$, there is a unique filtration of $E$ in $\mathcal Z$
\begin{equation}\label{eq14}
0 =: E_0 \subseteq E_1 \subseteq E_2 \subseteq \cdots \subseteq E_{n} \subseteq E_{n+1} := E
\end{equation}
where for all $1 \leq i \leq n+1$, we have
\[
E_i \in \mathcal T_i  \text{\quad and \quad} E_i/E_{i-1} \in \mathcal T_i \cap \mathcal F_{i-1}.
\]
Besides, the categories $\mathcal T_i \cap \mathcal F_{i-1}$ are extension-closed and have pairwise trivial intersections.
\end{theorem}

\begin{proof}
We construct the filtration \eqref{eq14} one term at a time, starting from the right-hand side. Note that $E_{n+1} = E \in \mathcal T_{n+1}$ by definition.

Using the torsion pair $(\mathcal T_n,\mathcal F_n)$, which is the same as $(\BB_n, \BB_n^\circ)$, we can write $E$ as an extension
\[
 0 \to E' \to E \to E'' \to 0
\]
where $E' \in \mathcal T_n$ and $E'' \in \mathcal F_n$.  We define $E_n := E'$, so that $E_n \in \mathcal T_n$.  Then $E_{n+1}/E_n = E/E' \cong E'' \in \mathcal F_n = \mathcal T_{n+1} \cap \mathcal F_n$.

From here on, for each $i$ in the sequence $n, n-1, \cdots, 2$, we use the torsion pair $(\mathcal T_{i-1}, \mathcal F_{i-1})$ to write $E_i$ as an extension
\[
 0 \to E_i' \to E_i \to E_i'' \to 0
\]
where $E_i' \in \mathcal T_{i-1}$ and $E_i'' \in \mathcal F_{i-1}$.  Then we define $E_{i-1} := E_i'$, so $E_{i-1} \in \mathcal T_{i-1}$, while $E_i/E_{i-1} \in \mathcal F_{i-1}$.  Since $E_i \in \mathcal T_i$ and $\mathcal T_i$ is closed under quotients in $\mathcal Z$, we have $E_i/E_{i-1} \in \mathcal T_i \cap \mathcal F_{i-1}$.

Since  the categories $\mathcal T_i, \mathcal F_i$ are extension-closed for all $i$, the intersection $\mathcal T_i \cap \mathcal F_{i-1}$ is also extension-closed for all $i$.

For any $1 \leq i < j \leq n+1$, we have $\mathcal T_i \subseteq \mathcal T_{j-1}$ from the definition of these categories.  On the other hand, $T_{j-1}$ and $\mathcal T_j \cap \mathcal F_{j-1}$ have trivial intersection.  Hence $\mathcal T_i \cap \mathcal F_{i-1}$ and $\mathcal T_j \cap \mathcal F_{j-1}$ have trivial intersection.  Lastly, the uniqueness of such a filtration follows from the fact that, with respect to any torsion pair, every object in $\mathcal Z$ has a unique filtration by its torsion part and torsion-free part.
\end{proof}

Now, we check that our filtration \eqref{eq14} coincides with Jensen-Madsen-Su's filtration \eqref{eq3} when $\mathcal Z$ has homological dimension 2 and the spectral condition holds.
We begin with the observation that the category $\KK_0$ in Section \ref{sec3} has a simpler description:

\begin{lemma}\label{lemma18}
Suppose $\mathcal Z$ is an abelian category of homological dimension 2 and the spectral condition holds.  Then any quotient $E$ of an object in $\XX_0$ lies in $\KK_0$, i.e.\ $E$ fits in a short exact sequence in $\mathcal Z$
\[
 0 \to A \to B \to E \to 0
\]
where $A \in \XX_2$ and $B \in \XX_0$.  As a result,
\begin{align*}
  \KK_0 &= \{ \text{$\mathcal Z$-quotients of objects in $\XX_0$}\}, \text{ and} \\
  \EE_0 &= [\XX_0] = [\BB_1 \cap \BB_2].
\end{align*}
\end{lemma}

\begin{proof}
Suppose we have a short exact sequence in $\mathcal Z$
\begin{equation}\label{eq12}
0 \to C \to D \to E \to 0
\end{equation}
where $D \in \XX_0$.  Since $\XX_0 = \BB_1 \cap \BB_2$ when $\mathcal Z$ has  homological dimension two, and $\BB_2$ is closed under quotients in $\mathcal Z$, we have $E \in \BB_2$.

On the other hand, applying $\Phi$ to \eqref{eq12} and taking the long exact sequence of cohomology, we see that $\Phi^1 E \cong \Phi^2 C$.  By Lemma \ref{lemma7}, $\Psi^p\Phi^1 E=0$ for all $p \neq -2$.  Thus $E \cong E^{0,0}_\infty$, and the spectral sequence \eqref{eq1} gives us a short exact sequence in $\mathcal Z$
\[
 0 \to \Psi^{-2}\Phi^1 E \to \Psi^0 \Phi^0 E \to E \to 0.
\]
By Lemma \ref{lemma9}, we see that $E$ lies in $\KK_0$.  The rest is clear.
\end{proof}



\begin{lemma}\label{lemma19}
When $\mathcal Z$ has homological dimension 2 and the spectral condition holds, we have $\BB_2 \cap \EE_0^\circ = \XX_1$.
\end{lemma}

\begin{proof}
Clearly, $\XX_1 \subseteq \BB_2$.  Given any $E \in \XX_1$, we also want to show $E \in \EE_0^\circ$.  Since $\EE_0$ is the extension closure of $\KK_0$,  it suffices to show that any morphism $G \overset{\al}{\to} E$ in $\mathcal Z$, where $G \in \KK_0$, is the zero morphism.  However, by the definition of $\KK_0$, we have a surjection $G' \overset{\beta}{\twoheadrightarrow} G$ in $\mathcal Z$ where $G' \in \XX_0$.  Since $\Hom (\XX_0,\XX_1)=0$, $\al \beta$ must be zero, and so $\al$ itself must be zero.  Hence $\XX_1 \subseteq \BB_2 \cap \EE_0^\circ$.

Conversely, suppose $E \in \BB_2 \cap \EE_0^\circ$.  By applying $\Psi$ to the exact triangle $\Phi^0E \to \Phi E \to \Phi^1 E [-1]$ and taking the long exact sequence of cohomology (or, equivalently, by using the spectral sequence \eqref{eq1}), we get an exact sequence in $\mathcal Z$
\[
 0 \to \Psi^{-2}\Phi^1 E \to \Psi^0 \Phi^0 E \overset{\gamma}{\to} E \to \Psi^{-1}\Phi^1 E \to 0.
\]
Since $\Psi^0 \Phi^0E \in \XX_0$ by Lemma \ref{lemma9} and $E \in \EE_0^\circ$ (and $\XX_0 \subseteq \EE_0$), we have $\gamma=0$.  Hence $\Psi^{-2}\Phi^1 E \cong \Psi^0\Phi^0E$, forcing both these two terms to be zero by Lemma \ref{lemma9} and since $\XX_2 \cap \XX_0=0$.  Hence $E \cong \Psi^{-1}\Phi^1 E$; by Lemma \ref{lemma12}, we have $E \in \XX_1$, as wanted.
\end{proof}

\begin{coro}\label{coro5}
The filtration \eqref{eq14} coincides with the filtration \eqref{eq3} when $\mathcal Z$ has homological dimension 2 and the spectral condition holds.
\end{coro}

\begin{proof}
When $n=2$ in \eqref{eq14}, we have $E/E_2 \in \mathcal T_3 \cap \mathcal F_2 = \BB_2^\circ=\EE_2$ by Lemma \ref{lemma15}.  Also, $E_2/E_1 \in \mathcal T_2 \cap \mathcal F_1 = \BB_2 \cap [\XX_0]^\circ = \BB_2 \cap \EE_0^\circ = \XX_1$ by Lemmas \ref{lemma18} and  \ref{lemma19}.  Finally,  $E_1 \in \mathcal T_1 = [\XX_0]=\EE_0$.  Hence the filtration \eqref{eq14}  indeed reduces to \eqref{eq3} under our hypotheses.
\end{proof}

\end{document}